\documentclass[12pt,reqno]{amsart}

\numberwithin{equation}{section}

\usepackage{amssymb}
\usepackage{fullpage}			
\usepackage{xcolor}
\usepackage{booktabs}

\usepackage{hyperref} 
\definecolor{mycitecolor}{rgb}{1,0,0}
\definecolor{mylinkcolor}{rgb}{0.66,0,0} 	
\definecolor{myurlcolor}{rgb}{0.33,0,0}
\hypersetup{colorlinks=true,citecolor=mycitecolor,linkcolor=mylinkcolor,urlcolor=myurlcolor,breaklinks=true}

\usepackage[capitalize]{cleveref}	

\usepackage{array}
\newcolumntype{P}[1]{>{\centering\arraybackslash}p{#1}}

\theoremstyle{plain}
\newtheorem{theorem}{Theorem}[section]

\newtheorem{corollary}[theorem]{Corollary}

\newtheorem{lemma}[theorem]{Lemma}

\newtheorem{proposition}[theorem]{Proposition}

\theoremstyle{definition}

\newtheorem{problem}[theorem]{Problem}

\newtheorem*{remark*}{Remark}

\crefname{figure}{Figure}{Figures}			
\crefname{equation}{}{}					

\newcommand{\ldiv}{\backslash}
\newcommand{\rdiv}{/}
\newcommand{\inv}{^{-1}}
\newcommand{\exponent}[1]{\mathrm{exp}(#1)}
\newcommand{\aut}[1]{\mathrm{Aut}(#1)}
\newcommand{\mlt}[1]{\mathrm{Mlt}(#1)}
\newcommand{\rmlt}[1]{\mathrm{Mlt}_r(#1)}
\newcommand{\lmlt}[1]{\mathrm{Mlt}_\ell(#1)}
\newcommand{\linn}[1]{\mathrm{Inn}_\ell(#1)}
\newcommand{\rinn}[1]{\mathrm{Inn}_r(#1)}
\newcommand{\inn}[1]{\mathrm{Inn}(#1)}
\newcommand{\nuc}[1]{\mathrm{Nuc}(#1)}
\newcommand{\lnuc}[1]{\mathrm{Nuc}_\ell(#1)}
\newcommand{\rnuc}[1]{\mathrm{Nuc}_r(#1)}
\newcommand{\mnuc}[1]{\mathrm{Nuc}_m(#1)}
\newcommand{\com}[1]{\mathrm{Com}(#1)}
\newcommand{\lt}{L}
\newcommand{\rt}{R}
\newcommand{\tm}{T}
\newcommand{\Z}{C}

\def\smallsudoku#1#2#3#4#5#6#7#8#9
{
\begingroup
\addtolength{\tabcolsep}{-1.8pt}
\renewcommand{\arraystretch}{0.82}
\begin{small}
\begin{tabular}{|@{\hspace{0.4em}}ccc@{\hspace{1.4em}}ccc@{\hspace{1.4em}}ccc@{\hspace{0.4em}}|}
\hline
\rule{0pt}{2.4ex}#1\\
#2\\
#3\\
\rule{0pt}{4ex}#4\\
#5\\
#6\\
\rule{0pt}{4ex}#7\\
#8\\
#9\\
\hline
\end{tabular}
\end{small}
\endgroup
}

\begin{document}

\title{Bol loops of order 27}

\author{Alexander Grishkov}
\author{Michael Kinyon}
\author{Petr Vojt\v echovsk\'y}

\address[Kinyon,Vojt\v{e}chovsk\'y]{Department of Mathematics\\ University of Denver\\ 2390 S.~York St.\\ Denver, CO 80208, USA}
\address[Grishkov]{Institute of Mathematics and Statistics, University of S\~ao Paulo, Brazil and Sobolev Institute of Mathematics, Omsk, Russia}

\email[Grishkov]{shuragri@gmail.com}
\email[Kinyon]{michael.kinyon@du.edu}
\email[Vojt\v{e}chovsk\'y]{petr.vojtechovsky@du.edu}

\thanks{A.~Grishkov supported by FAPESP (grant 2018/23690-6), CNPq (grants 307593/2023-1 and 406932/2023-9), Brazil, and in accordance with the state task of the IM SB RAS, project FWNF-2022-003, Russia. P.~Vojt\v{e}chovsk\'y supported by the Simons Foundation Mathematics and Physical Sciences Collaboration Grant for Mathematicians MPS-TSM-00855097.}



\begin{abstract}
We classify Bol loops of order $27$, using a combination of theoretical results and computer search. There are $15$ Bol loops of order $27$, including five groups. New constructions for the ten nonassociative Bol loops of order $27$ are given.
\end{abstract}

\maketitle

\section{Introduction}

Classifying algebras up to isomorphism in a given variety is an important task in abstract algebra. In this paper we will classify Bol loops of order $27$. The classification is made complicated by the fact that, unlike for $p$-groups or Moufang $p$-loops, Bol $p$-loops are not necessarily centrally nilpotent.

It turns out that all $15$ Bol loops of order $27$ have already appeared in the literature. Our result can therefore be summarized as follows: \emph{No additional Bol loops of order $27$ exist}. We also present compact constructions of all Bol loops of order $27$ and list some invariants by which the loops can be recognized.

\medskip

A loop is \emph{right Bol} if it satisfies the identity
\begin{equation}\label{Eq:RBol}
	((x\cdot y)\cdot z)\cdot y = x\cdot ((y\cdot z)\cdot y).
\end{equation}
Left Bol loops are loops satisfying the dual (mirror image) of the identity \cref{Eq:RBol}. The much-studied Moufang loops are precisely the loops that are both right Bol and left Bol.

\begin{remark*}
We will work with right Bol loops here. All results obtained in this paper for right Bol loops can be dualized to left Bol loops. It is customary to refer to right Bol loops or left Bol loops simply as \emph{Bol loops}, as we have done in the title and in the narrative above. However, we employ this abbreviation only when it is safe to do so. We therefore avoid statements such as ``Bol loops have the right inverse property,'' which is meant to be a shorthand for the true statement ``right Bol loops have the right inverse property,'' but which can potentially be read as one of the false statements ``left Bol loops have the right inverse property'' or ``Bol loops have the inverse property."
\end{remark*}

\subsection{Related classification results for Bol loops}

Bol loops were introduced by Bol \cite{Bol} in the context of finite geometry. The first systematic algebraic study of Bol loops is due to Robinson \cite{Robinson_thesis,Robinson}, who proved, among other results, that any Bol loop of prime order is a group.

Let $p\ne q$ be primes. Burn showed that Bol loops of order $p^2$ and $2p$ are groups and he classified nonassociative Bol loops of order $8$ \cite{BurnI}. Kinyon, Nagy and Vojt\v{e}chovsk\'y classified Bol loops of order $pq$ \cite{KNVpq}, building upon \cite{NiederreiterRobinson_pq}. (See \cite{Vojtechovsky} for a classification of Bol loops of order $pq$ up to isotopism.) For every odd prime $p$, there are precisely two nonassociative right Bol loops of order $2p^2$ \cite{BurnIII,SharmaSolarin_2p2}.

There are many interesting examples of finite simple non-Moufang Bol loops \cite{BaumeisterStein,Nagy_simple,Nagy_simple2}, making the classification of finite simple Bol loops challenging.

\emph{Right Bruck loops} are right Bol loops satisfying the property $(xy)^{-1}=x^{-1}y^{-1}$. Bruck loops of order $p^3$ (resp. $pq$) were classified in \cite{BiancoBonatto} (resp. \cite{KNVpq}).

\subsection{Bol loops of order $27$ in the literature}

Let us briefly describe where all Bol loops of order $27$ appeared for the first time as far as we know.

For every prime $p$, there are precisely $5$ groups of order $p^3$. For an odd prime $p$, this was proved independently by Cole and Glover \cite{ColeGlover}, H\"older \cite{Holder} and Young \cite{Young}.

Using a backtracking algorithm, Moorhouse \cite{Moorhouse} constructed many Bol loops of small orders, including the $8$ nonassociative Bol loops of order $27$ with a nontrivial center.

One of the two nonassociative Bol loops of order $27$ with trivial center was found by Keedwell \cite{Keedwell1963,Keedwell1965} in a different context---the cited papers do not mention that the loop satisfies a Bol identity.
Kinyon noticed that Keedwell's loop $Q$ is a Bol loop and discovered the second Bol loop with trivial center by investigating all loop isotopes of $Q$. The two loops then appeared in \cite{FoguelKinyon} and \cite{KinyonWanless}.

\section{Loops}

See \cite{Bruck} for an introduction to loop theory. A \emph{loop} $(Q,\cdot,\ldiv,\rdiv,1)$ is a set $Q$ with binary operations $\cdot$, $\ldiv$, $\rdiv$ and an element $1\in Q$ satisfying the identities $x\cdot(x\ldiv y) = y = x\ldiv(x\cdot y)$, $(x\cdot y)\rdiv y = x = (x\rdiv y)\cdot y$ and $1\cdot x=x = x\cdot 1$. The theory of loops encompasses the vast space between purely combinatorial objects (normalized latin squares) and highly structured algebras (groups). On the algebraic side, one often studies loops satisfying additional axioms.

\begin{remark*}
From now on, we will also use juxtaposition in place of the multiplication operation $\cdot$, and we declare $\cdot$ to be less binding than the division operations $\rdiv$ and $\ldiv$, which will in turn be less binding than the juxtaposition. So, for instance, $x/yz\cdot u\ldiv v$ stands for $(x/(y\cdot z))\cdot (u\ldiv v)$.
\end{remark*}

For a loop $Q$ and $x\in Q$, let $\lt_x:Q\to Q$, $y\mapsto \lt_x(y)=xy$ be the \emph{left translation} by $x$ in $Q$, and $\rt_x:Q\to Q$, $y\mapsto \rt_x(y)=yx$ the \emph{right translation} by $x$ in $Q$. Denote by
\begin{displaymath}
	\lmlt{Q} = \langle \lt_x:x\in Q\rangle,\quad \rmlt{Q} = \langle \rt_x:x\in Q\rangle,\quad \mlt{Q} = \langle \rt_x,\lt_x:x\in Q\rangle
\end{displaymath}
the \emph{left multiplication group}, the \emph{right multiplication group} and the \emph{multiplication group} of $Q$, respectively.

The \emph{left inner mapping group} $\linn{Q}$,  the \emph{right inner mapping group} $\rinn{Q}$ and the \emph{inner mapping group} $\inn{Q}$ are then the stabilizers of $1$ in $\lmlt{Q}$, $\rmlt{Q}$ and $\mlt{Q}$, respectively. With
\begin{displaymath}
	\lt_{x,y} = \lt_{xy}^{-1}\lt_x\lt_y,\quad \rt_{x,y} = \rt_{yx}^{-1}\rt_x\rt_y,\quad \tm_x=\lt_x\inv\rt_x,
\end{displaymath}
it is well known that $\linn{Q}=\langle \lt_{x,y}:x,y\in Q\rangle$, $\rinn{Q} = \langle \rt_{x,y}:x,y\in Q\rangle$ and $\inn{Q}=\langle \lt_{x,y},\rt_{x,y},\tm_x:x,y\in
Q\rangle$.

A subloop $N\le Q$ is \emph{normal} in $Q$, denoted by $N\unlhd Q$, if $\varphi(N)=N$ for all $\varphi\in\inn{Q}$. The \emph{factor loop} $Q/N$ is then defined as usual.

For a loop $Q$ consider
\begin{align*}
	\lnuc{Q} &= \{x\in Q:x\cdot yz = xy\cdot z\text{ for all }y,z\in Q\},\\
	\mnuc{Q} &= \{x\in Q:y\cdot xz = yx\cdot z\text{ for all }y,z\in Q\},\\
	\rnuc{Q} &= \{x\in Q:y\cdot zx = yz\cdot x\text{ for all }y,z\in Q\},\\
	\nuc{Q} &= \lnuc{Q}\cap\mnuc{Q}\cap\rnuc{Q},
\end{align*}
the \emph{left nucleus}, \emph{middle nucleus}, \emph{right nucleus} and \emph{nucleus} of $Q$, respectively. Each of the four nuclei is a subloop of $Q$. Note that the elements of $\lnuc{Q}$ are precisely the fixed points of $\rinn{Q}$. Let also
\begin{align*}
	\com{Q} &= \{x\in Q:xy = yx\text{ for all }y\in Q\},\\
	Z(Q) &= \nuc{Q}\cap\com{Q}
\end{align*}
be the \emph{commutant} and the \emph{center} of $Q$, respectively. The commutant is not necessarily a subloop of $Q$, even in Bol loops \cite{KPV_comm}. The center is always a normal subloop of $Q$. A loop $Q$ is said to be \emph{centrally nilpotent} if the series
\begin{displaymath}
	Q,\ Q/Z(Q),\ (Q/Z(Q))/Z(Q/Z(Q)),\ \dots
\end{displaymath}
reaches the trivial loop in finitely many steps.

Let $Z$ be an abelian group and $F$ a loop. A loop $Q$ is a \emph{central extension} of $Z$ by $F$ if $Z\le Z(Q)$ and $Q/Z$ is isomorphic to $F$. It is well known that up to isomorphism, all central extensions of $Z=(Z,+,0)$ by $F=(F,\cdot,1)$ are obtained by modifying the direct product $Z\times F$ as $(a,x)*(b,y) = (a+b+\theta(x,y),xy)$, where $\theta:F\times F\to Z$ is a \emph{loop cocycle}, that is, a mapping satisfying $\theta(1,x)=\theta(y,1)=0$ for all $x,y\in F$.

A loop $Q$ is \emph{power associative} if for every $x\in Q$ the subloop $\langle x\rangle$ of $Q$ is associative, that is, $\langle x\rangle$ is a group. In a power associative loop $Q$, we can safely use the notation $x^n$ to denote powers of $x\in Q$, with $n$ any integer. In particular, $x^{-1}$ is the two-sided inverse of $x\in Q$.

A loop $Q$ has the \emph{right inverse property} if it has two-sided inverses and satisfies $(xy)y^{-1}=x = (xy^{-1})y$. A power associative loop $Q$ is \emph{right power alternative} if $(xy^n)y^m = xy^{n+m}$ for all $x,y\in Q$ and $n,m\in\mathbb Z$. If $Q$ is a finite right power alternative loop, then the order $|x|$ of $x\in Q$ divides $|Q|$.

Let $p$ be a prime. A finite loop $Q$ is a \emph{$p$-loop} if $|Q|=p^n$ for some integer $n$.

For a loop $Q$, the \emph{derived subloop} $Q'$ is the smallest normal subloop $N$ of $Q$ such that $Q/N$ is an abelian group. A loop $Q$ is \emph{solvable} if the \emph{derived series}
\begin{displaymath}
	Q\ge Q'\ge Q''\ge\cdots
\end{displaymath}
reaches the trivial loop in finitely many steps. (See \cite{StaVoj} for an alternative definition of solvability for loops based on the commutator theory of universal algebra.)

\section{Constructions for nonassociative Bol loops of order $27$}\label{Sc:Constructions}

In this section we construct the ten nonassociative right Bol loops $B_1,\dots,B_{10}$ of order $27$ so that we can refer to them later. The multiplication tables of the ten loops can be found on the website of the third author. Most calculations used to discover these constructions were performed in the \texttt{GAP} \cite{GAP} package \texttt{RightQuasigroups} \cite{RightQuasigroups}.

\begin{table}
\begin{footnotesize}
\begin{tabular}{lrrrrrrrrrr}
\toprule
$Q$                                 	&$B_1$ 	&$B_2$	&$B_3$ 	&$B_4$	&$B_5$	&$B_6$	&$B_7$	&$B_8$	&$B_9$	  &$B_{10}$\\
$|Z(Q)|$						        &3		&3     	&3		&3		&3		&3		&3		&3		&1		  &1\\					
$\exponent{Q}$                          &9      &9      &9      &9     	&9     	&9      &9     	&9     	&3		  &3\\
$|\{x\in Q:|x|=3\}$                     &2      &14     &8   	&2     	&20  	&14 	&8		&2		&26		  &26\\
$|Q'|$                              	&3      &3      &3    	&3		&3		&3		&3		&3		&9		  &9\\
$|\lnuc{Q}|$                       	    &9    	&9		&9		&9		&9		&9		&9		&9		&9		  &9\\
$\exponent{\lnuc{Q}}$                   &9		&3		&9		&9		&9		&3		&9		&9		&3		  &3\\
$|\{(x,y)\in Q\times Q:xy=yx\}|$        &459	&459	&459	&459   	&405	&405	&405	&405	&153	  &153\\
$|\rmlt{Q}|$                        	&81		&81		&81		&81		&81		&81		&81		&81		&243	  &243\\
$|\lmlt{Q}|$                        	&243    &243	&243	&243	&243	&243	&243	&243	&139968   &139968\\
$|\mlt{Q}|$                         	&2187	&2187	&2187	&2187	&2187	&2187	&2187	&2187	&139968   &139968\\
$|\aut{Q}|$                         	&54     &18		&18		&27		&108	&36		&36		&54		&72		  &144\\
is $Q$ right Bruck?                	    &no     &no     &no     &no     &yes    &yes    &yes    &yes    &no       &no\\
associated right Bruck loop             &$B_5$  &$B_6$  &$B_7$  &$B_8$  &$B_5$  &$B_6$  &$B_7$  &$B_8$  &$\Z_3^3$ &$\Z_3^3$\\
\bottomrule
\end{tabular}
\end{footnotesize}
\medskip
\caption{Some invariants of the ten nonassociative right Bol loops of order $27$.}\label{Tb:Invariants}
\end{table}

Table \ref{Tb:Invariants} summarizes some invariants of these loops. Note that any two loops in the table can be distinguished by, for instance, the size of their automorphism group and the number of elements of order $3$ they contain. Recall that with any right Bol loop $Q$ in which $x\mapsto x^2$ is a permutation we can associate a right Bruck loop $(Q,\circ)$ given by $x\circ y = ((xy^2)x)^{1/2}$. The last row of the table reports the isomorphism type of the associated right Bruck loop. Here, $\Z_n=\{0,\dots,n-1\}$ denotes the cyclic group of order $n$ and $\Z_3^3=\Z_3\times\Z_3\times\Z_3$. We observed computationally that each pair $(B_1,B_5)$, $(B_2,B_6)$, $(B_3,B_7)$, $(B_4,B_8)$ and $(B_9,B_{10})$ consists of isotopic loops. No other loops $B_i$, $B_j$ with $i\ne j$ are isotopic.

\begin{table}
\begin{footnotesize}
\begin{tabular}{lrrrrr}
\toprule
$Q$                                 	&$B_1$ 					&$B_2$						&$B_3$ 						&$B_4$				&$B_5$\\
$\exponent{G}$				&9						&9							&9							&9					&9\\
$|Z(G)|$						&3						&3							&3							&3					&3\\
\texttt{GAP} id of $G$		&[81,9]					&[81,7]						&[81,8]						&[81,10]			&[81,9]\\	
structural information			&$(\Z_9\times\Z_3):\Z_3$&$\Z_3^3:\Z_3$				&$(\Z_9\times\Z_3):\Z_3$		&$\Z_3^2.\Z_3^2$	&$(\Z_9\times\Z_3):\Z_3$\\
\\
$Q$                                 	&$B_6$					&$B_7$						&$B_8$						&$B_9$				&$B_{10}$\\
$\exponent{G}$				&9						&9							&9							&3					&3\\
$|Z(G)|$						&3						&3							&3							&9					&9\\
\texttt{GAP} id of $G$		&[81,7]					&[81,8]						&[81,10]					&[243,37]			&[243,37]\\	
structural information			&$\Z_3^3:\Z_3$			&$(\Z_9\times\Z_3):\Z_3$		&$\Z_3^2.\Z_3^2$			&					&\\
\bottomrule
\end{tabular}
\end{footnotesize}
\medskip
\caption{Some invariants of the right multiplication groups $G=\rmlt{Q}$ of the ten nonassociative right Bol loops of order $27$.}\label{Tb:Mltrs}
\end{table}

Since the structure of right multiplication groups plays an important role in the theory of right Bol loops, we collect some additional information about the right multiplication groups of the ten nonassociative loops $B_1,\dots,B_{10}$ in Table \ref{Tb:Mltrs}. The row of the table with structural information uses notational conventions of \texttt{GAP}.

Note that each of the loops $B_1,\dots,B_8$ has center isomorphic to $\Z_3$ and can therefore be constructed from the factor $Q/Z(Q)$ and a loop cocycle, a $9\times 9$ table with entries in $\Z_3$. However, we opt for alternative constructions, some of which will be useful also in the cases $B_9$ and $B_{10}$ where the center is trivial.

\subsection{The six loops with left nucleus of exponent nine}

For parameters $x,y\in\Z_9^*$ and $r\in\Z_9$, consider the magma $Q(x,y,r)$ defined on $\Z_3\times\Z_9$ by the multiplication formula
\begin{displaymath}
	(u,i)(v,j) = \left(u+v,\ i + f(u,v)j + r\left\lfloor\frac{u+v}{3}\right\rfloor \right),
\end{displaymath}
where for the purposes of the floor function we understand $u$ and $v$ as integers in $\{0,1,2\}$ and where $f:\Z_3\times\Z_3\to\Z_9^*$ is given by
\begin{align*}
	&f(0,0)=1,	&&f(0,1)=1,		&&f(0,2)=1,\\
	&f(1,0)=x,	&&f(1,1)=1/y,	&&f(1,2)=y/x,\\
	&f(2,0)=y,	&&f(2,1)=x/y,	&&f(2,2)=1/x.
\end{align*}
Then
\begin{align*}
	&B_1=Q(1,7,0), &&B_3=Q(1,4,0), &&B_4=Q(1,7,3),\\
	&B_5=Q(4,4,0), &&B_7=Q(7,7,0), &&B_8=Q(4,4,3).
\end{align*}

\subsection{The two loops of exponent nine with left nucleus of exponent three}

Let us start with an auxiliary construction that will be useful here and in the next subsection.

 Let $k\in\Z_9$ and let $K$ be a $3\times 3$ matrix containing every element of $\Z_9$. Then $T(k,K)$ is the $3\times 3$ matrix such that
\begin{itemize}
\item the top left entry is equal to $k$,
\item the top row is a cyclic shift of a row of $K$,
\item every column is a cyclic shift of one of $(0,1,2)$, $(3,4,5)$ and $(6,7,8)$.
\end{itemize}
The assumption that $K$ contains all elements of $\Z_9$ guarantees that $T(k,K)$ is well-defined. (Note that $T(k,K)$ is invariant under permutations of rows of $K$.)
For instance, if
\begin{displaymath}
	k=5\quad\text{and}\quad K = \begin{array}{ccc}2&6&1\\0&4&8\\7&5&3\end{array}
\end{displaymath}
then
\begin{displaymath}
	T(k,K) = \begin{array}{ccc}5&3&7\\3&4&8\\4&5&6\end{array}.
\end{displaymath}

Let $M$, $N$ be two $9\times 9$ matrices with entries in $\Z_9$. Let us view $N$ as a block matrix with blocks $N_{uv}$ of size $3$, where $u,v\in\Z_3$. Suppose that every block $N_{uv}$ contains all elements of $\Z_9$. Let $T(M,N)$ be the $27\times 27$ matrix in which the $3\times 3$ block in (block) row $i\in\Z_9$ and (block) column $j\in\Z_9$ is equal to $T(M_{ij},N_{\lfloor i/3\rfloor\,\lfloor j/3\rfloor})$. We are done with the auxiliary construction.

\medskip

In this subsection, we will specialize to the situation when every $3\times 3$ block of $N$ is the matrix
\begin{displaymath}
	K=\begin{array}{ccc}0&1&2\\3&4&5\\6&7&8
	\end{array},
\end{displaymath}
in which case we will denote the matrix $T(k,K)$ just by $T(k)$ and the matrix $T(M,N)$ just by $T(M)$. (Note that $T(k)$ is a reversed circulant matrix with entries in one of $\{0,1,2\}$, $\{3,4,5\}$ and $\{6,7,8\}$.)

Consider the magma $Q(M)$ on $C_{27}$ whose multiplication table is obtained from the matrix $T(M)$ by adding
\begin{equation}\label{Eq:Z3}
	9((\lfloor i/9\rfloor + \lfloor j/9\rfloor)\text{ mod }3)
\end{equation}
to the entry in row $i\in\Z_{27}$ and column $j\in\Z_{27}$. In effect, $Q(M)$ has a coarse block structure of the cyclic group $\Z_3$ and its fine behavior is governed by $3\times 3$ reversed circulant matrices, each one arising from a single entry of $M$.

Let
\begin{center}
\begin{tabular}{P{10mm}P{50mm}P{10mm}P{10mm}P{50mm}}
$M_2\ = $
&\smallsudoku{0&3&6&0&3&6&0&3&6}{3&6&0&3&6&0&6&1&5}{6&0&3&6&0&3&5&7&0}{0&3&6&0&7&4&1&6&3}{3&6&0&6&3&2&4&0&6}	{6&0&3&5&1&8&7&3&0}{0&8&3&1&5&6&1&8&4}{3&0&8&8&0&4&8&3&2}{6&4&1&4&8&0&4&2&7}
&
&
$M_6\ = $
&\smallsudoku{0&3&6&0&3&6&0&3&6}{3&6&0&3&7&2&6&2&4}{6&0&3&7&0&5&4&8&0}{0&4&6&0&8&4&1&6&4}{3&8&2&6&3&0&4&0&7}{6&0&4&3&1&8&6&5&0}{0&7&3&1&4&6&1&8&3}{3&0&7&6&0&5&7&3&2}{6&5&2&4&7&0&4&1&7}
\end{tabular}
\end{center}
\medskip
Then $B_2=Q(M_2)$ and $B_6 = Q(M_6)$.

\subsection{The two loops with trivial center}

We will again use the auxiliary construction $T(M,N)$ but this time with a nontrivial matrix $N$. Let $M$, $N$ be two $9\times 9$ matrices with entries in $\Z_9$, where again every block $N_{uv}$ contains all elements of $\Z_9$. Let $Q(M,N)$ be the magma whose multiplication table is obtained from the matrix $T(M,N)$ by adding \eqref{Eq:Z3} to the entry in row $i\in\Z_{27}$ and column $j\in\Z_{27}$.

Consider the matrices
\begin{center}
\begin{tabular}{P{10mm}P{50mm}P{10mm}P{10mm}P{50mm}}
$M_9\ = $
&\smallsudoku{0&3&6&0&3&6&0&3&6}{3&6&0&5&8&2&3&6&0}{6&0&3&7&1&4&6&0&3}{0&1&2&0&6&3&0&2&1}{3&4&5&4&1&7&4&3&5}{6&7&8&8&5&2&8&7&6}{0&8&4&0&2&1&0&8&4}{3&2&7&3&5&4&5&1&6}{6&5&1&6&8&7&7&3&2}
&
&
$N_9\ = $
&\smallsudoku{0&1&2&0&1&2&0&1&2}{3&4&5&3&4&5&3&4&5}{6&7&8&6&7&8&6&7&8}{0&3&6&0&5&7&0&8&4}{1&4&7&1&3&8&1&6&5}{2&5&8&2&4&6&2&7&3}{0&7&5&0&6&3&0&5&7}{1&8&3&1&7&4&1&3&8}{2&6&4&2&8&5&2&4&6}
\end{tabular}
\end{center}
and
\begin{center}
\begin{tabular}{P{10mm}P{50mm}P{10mm}P{10mm}P{50mm}}
$M_{10}\ = $
&\smallsudoku{0&3&6&0&3&6&0&3&6}{3&6&0&5&8&2&8&2&5}{6&0&3&7&1&4&4&7&1}{0&1&2&0&4&8&0&2&1}{3&4&5&6&1&5&4&3&5}{6&7&8&3&7&2&8&7&6}{0&3&6&0&2&1&0&4&8}{3&6&0&8&7&6&6&1&5}{6&0&3&4&3&5&3&7&2}
&
&
$N_{10}\ = $
&\smallsudoku{0&1&2&0&1&2&0&1&2}{3&4&5&3&4&5&3&4&5}{6&7&8&6&7&8&6&7&8}{0&4&8&0&5&7&0&8&4}{1&5&6&1&3&8&1&6&5}{2&3&7&2&4&6&2&7&3}{0&4&8&0&4&8&0&5&7}{1&5&6&1&5&6&1&3&8}{2&3&7&2&3&7&2&4&6}
\end{tabular}
\end{center}
\medskip
Then $B_9=Q(M_9,N_9)$ and $B_{10}=Q(M_{10},N_{10})$.

\section{Bol loops: Theoretical results}\label{Sc:Theory}

We will collect several well known results for Bol loops and establish a few new results. Additional results for Bol loops (that we do not need here) are summarized in \cite{FoguelKinyon}.

Right Bol loops are right power alternative \cite{Robinson_thesis}. In particular, right Bol loops are power associative and have the right inverse property.

Let $p$ be a prime. As we have already mentioned in the introduction, Bol loops of order $p$ and $p^2$ are groups \cite{BurnI,Robinson_thesis}. While $p$-groups and Moufang $p$-loops are centrally nilpotent \cite{Drapal,Glauberman,GlaubermanWright}, Bol $p$-loops are not necessarily centrally nilpotent. If $p$ is an odd prime then a finite Bol loop $Q$ is a $p$-loop (that is, $|Q|$ is a $p$-power) if and only if $|x|$ is a $p$-power for every $x\in Q$.

\subsection{Basic structure of Bol loops of order $p^3$}

We start with a well known result whose proof we provide for the sake of completeness. Note that Proposition \ref{Pr:CyclicFactor} applies to Bol loops since Bol loops are power associative.

\begin{proposition}\label{Pr:CyclicFactor}
Let $Q$ be a power associative loop such that $Q/Z(Q)$ is a cyclic group. Then $Q$ is an abelian group.
\end{proposition}
\begin{proof}
We have $Q/Z(Q)=\langle aZ(Q)\rangle$ for some $a\in Q$. Since $(aZ(Q))^n = a^nZ(Q)$, every element of $Q$ can be written as $a^nz$ for some integer $n$ and $z\in Z(Q)$. Let us consider three arbitrary elements $a^mz_1$, $a^nz_2$ and $a^kz_3$ of $Q$ written in this form. Since central elements associate and commute with all elements of $Q$, we have $(a^mz_1\cdot a^nz_2)(a^kz_3) = (a^ma^n)a^k\cdot (z_1z_2)z_3$ and $(a^mz_1) (a^nz_2\cdot a^kz_3) = a^m(a^na^k)\cdot z_1(z_2z_3)$. By power associativity, $(a^ma^n)a^k = a^m(a^na^k)$. Of course, $(z_1z_2)z_3 = z_1(z_2z_3)$. Hence $Q$ is associative. Similarly, we have $a^mz_1\cdot a^nz_2 =a^ma^n\cdot z_1z_2 = a^na^m\cdot z_2z_1 = a^nz_2\cdot a^mz_1$, proving that $Q$ is commutative.
\end{proof}

\begin{theorem}\label{Th:Structure}
Let $p$ be a prime and let $Q$ be a Bol loop of order $p^3$. Then one of the following situations occurs:
\begin{enumerate}
\item[(i)] $Z(Q)=1$, or
\item[(ii)] $Q$ is an abelian group, or
\item[(iii)] $Z(Q)\cong C_p$ and $Q/Z(Q)\cong C_p\times C_p$.
\end{enumerate}
\end{theorem}
\begin{proof}
Since $Z(Q)\unlhd Q$, we have $|Z(Q)|\in\{1,p,p^2,p^3\}$. Suppose that $Z(Q)>1$. Then $|Q/Z(Q)|\in\{1,p,p^2\}$, so $Q/Z(Q)$ is group. If $Q/Z(Q)$ is cyclic then $Q$ is an abelian group by Proposition \ref{Pr:CyclicFactor}. The only remaining possibility is $Q/Z(Q)\cong C_p\times C_p$.
\end{proof}

\subsection{Solvability of Bol $p$-loops}

For odd primes $p$, Moufang $p$-loops are nilpotent, but Bol $p$-loops are not necessarily nilpotent. However, Nagy proved \cite[Lemma 5.1]{NagyCMUC}:

\begin{theorem}[Nagy]\label{Nagy}
Let $p$ be an odd prime and $Q$ a Bol $p$-loop. Then $Q$ is solvable.
\end{theorem}

\begin{corollary}\label{nns}
	Let $p$ be an odd prime and $Q$ a Bol $p$-loop of order bigger than $p$. Then $Q$ has a nontrivial proper normal subloop.
\end{corollary}
\begin{proof}
By Theorem \ref{Nagy}, $Q$ is solvable, so $Q'<Q$. If $Q'=1$ then $Q$ is an abelian $p$-group and we are done. Otherwise $1<Q'<Q$ is the sought after normal subloop.
\end{proof}

\begin{proposition}
Let $p$ be an odd prime and let $Q$ be a Bol loop of order $p^3$. Then $Q$ contains a normal subloop of order $p^2$.
\end{proposition}
\begin{proof}
By Corollary \ref{nns}, $Q$ contains a nontrivial proper normal subloop $H$. If $|H|=p^2$, we are done, so suppose that $|H|=p$. Then $|Q/H|=p^2$ and thus $Q/H$ is an abelian group. Hence there is $H/H<K/H<Q/H$ such that $K/H\unlhd Q/H$ and thus $K\unlhd Q$ and $|K|=p^2$ by the Correspondence Theorem.
\end{proof}

\subsection{The left nucleus and the right multiplication group}

Since $R_xR_yR_x$ is a right translation in right Bol loops, the right section $\{R_x:x\in Q\}$ is closed under the operation $(u,v)\mapsto uvu$. It then follows from \cite[Theorem 15]{Glauberman}:

\begin{theorem}[Glauberman]\label{Glaub}
Let $Q$ be a right Bol loop of odd order. Let $p$ be a prime. Then $p$ divides $|Q|$ if and only if $p$ divides $|\rmlt{Q}|$.
\end{theorem}

The following result was obtained by Foguel, Kinyon and Phillips in \cite[Theorem 6.4]{FKP}:

\begin{theorem}\label{Th:Lagrange}
Let $Q$ be a Bol loop of odd order and $H\le Q$. Then $|H|$ divides $|Q|$.
\end{theorem}

\begin{proposition}\label{Pr:LeftNuc}
Let $Q$ be a right Bol loop of odd order, let $p$ be the smallest prime dividing $|Q|$, and let $H>1$ be an $\rinn{Q}$-invariant subloop. Suppose that one of the following conditions holds:
\begin{enumerate}
	\item[(i)] $Q$ is a $p$-loop, or
	\item[(ii)] $|H| = p$.
\end{enumerate}
Then $H\cap \lnuc{Q} > 1$.
\end{proposition}
\begin{proof}
By Theorem \ref{Glaub}, the primes dividing $|Q|$ coincide with the primes dividing $|\rmlt{Q}|$. Since $H$ is $\rinn{Q}$-invariant and $1$ is a fixed point of $\rinn{Q}$, $H\backslash \{1\}$ is $\rinn{Q}$-invariant. Any $\rinn{Q}$-invariant subset of $Q$ is a union of $\rinn{Q}$-orbits.

In case (i), every orbit of $\rinn{Q}$ has size a power of $p$. Since $|Q|$ is a power of $p$, so is $|H|$, by Theorem \ref{Th:Lagrange}. Since $\rinn{Q}$ fixes $1$, $H\backslash \{1\}$ must contain at least $p-1$ additional fixed points of $\rinn{Q}$.

In case (ii), every nontrivial orbit of $\rinn{Q}$ has size at least $p$ since $p$ is the smallest prime dividing $|\rmlt{Q}|$. Since $|H\backslash \{1\}| = p-1$, $H\backslash \{1\}$ also consists of fixed points of $\rinn{Q}$.

In both cases the proof is completed by noting that $\lnuc{Q}$ is precisely the set of fixed points of $\rinn{Q}$.
\end{proof}

\subsection{Normal subloop of order $3$}

\begin{proposition}[\cite{KP_comm}, Theorem~1.1]\label{Prp:comm_sbl}
Let $Q$ be a finite Bol loop of odd order. Then $\com{Q}$ is a subloop of $Q$.
\end{proposition}

\begin{lemma}\label{Lem:rnuccomm}
Let $Q$ be a finite right Bol loop of odd order and let $H$ be a subgroup of $\lnuc{Q}$ of order $3$. Suppose that $T_x(H)\subseteq H$ for all $x\in Q$. Then $H\leq \com{Q}$.
\end{lemma}
\begin{proof}
Let $1\ne c\in H$ so that $H = \langle c\rangle$. Throughout the proof we will use $c\in\lnuc{Q}$. Since $T_x(H)=H$ and $T_x(1)=1$, we have $T_x(c) = c$ or $T_x(c) = c\inv$. Equivalently, for all $x\in Q$,
\begin{equation}\label{Eqn:or0}
	cx = xc\qquad\text{or}\qquad cx=xc\inv\,.
\end{equation}

Suppose that $ca\neq ac$ for some $a\in Q$, so that $ca=ac\inv$ by \eqref{Eqn:or0}. Then $ac\inv\cdot a = ca\cdot a = ca^2$ by the right alternative property, and hence $ca^2c\inv = (ac\inv\cdot a)c\inv = a(c\inv ac\inv) = a(c\inv ca) = a^2$ by the right Bol identity. Multiplying by $c$ on the right and using the right inverse property then yields $ca^2=a^2c$. We proved: for all $x\in Q$,
\begin{equation}\label{Eqn:or2}
 	cx = xc \qquad\text{or}\qquad cx^2 = x^2 c\,.
\end{equation}

We claim that, in fact, $cx^2=x^2c$ for all $x\in Q$. Suppose that $ca^2\ne a^2c$ for some $a\in Q$. By \eqref{Eqn:or2}, $ca=ac$. By \eqref{Eqn:or0} with $x=a^2$, $ca^2=a^2c^{-1}$, which yields $ca^2c=a^2$ by the right inverse property. Hence $a^2 = ca^2c = (ca\cdot a)c = (ac\cdot a)c = a(cac) = a(ac\cdot c) = a(ac^2)$ by the right power alternative and right Bol properties. Canceling $a$ on the left two times gives $c^2=1$, a contradiction. This establishes the claim.

Finally, since $Q$ has odd order, the mapping $Q\to Q$, $x\mapsto x^2$ is a bijection, so we have $cy = yc$ for all $y\in Q$. Therefore $c\in \com{Q}$.
\end{proof}

\begin{proposition}[\cite{KPV_comm}, Corollary~2.6]\label{Prp:KPV_comm}
Let $Q$ be a finite right Bol loop of odd order. Then $\com{Q}\cap \lnuc{Q} = Z(Q)$.
\end{proposition}

\begin{theorem}\label{Thm:centralsubloop}
 Let $Q$ be a finite Bol loop of odd order and let $H$ be a normal subloop of order $3$. Then $H\leq Z(Q)$.
\end{theorem}
\begin{proof}
Since the statement is self-dual, it suffices to prove it for a right Bol loop $Q$. Since $|H|=3$, Proposition \ref{Pr:LeftNuc} yields $H\leq \lnuc{Q}$.  We have $T_x(H)=H$ because $H\unlhd Q$. Then Lemma \ref{Lem:rnuccomm} gives $H\leq \com{Q}$. By Proposition \ref{Prp:KPV_comm}, $H\leq Z(Q)$.
\end{proof}

\begin{corollary}\label{Cr:Likewise}
Let $Q$ be a Bol loop of order $27$. If either $|Z(Q)|=3$ or $|Q'|=3$ then $Z(Q)=Q'$.
\end{corollary}
\begin{proof}
Suppose that $|Z(Q)|=3$. Then $Q/Z(Q)$ is an abelian group and hence $Q'\le Z(Q)$. Since $Q$ is not an abelian group, $1<Q'$. Hence $Q'=Z(Q)$.

Now suppose that $|Q'|=3$. Since $Q'\unlhd Q$, Theorem \ref{Thm:centralsubloop} implies $Q'\le Z(Q)$. If $|Z(Q)|>3$ then $Q$ is an abelian group by Theorem \ref{Th:Structure}, but then $Q'=1$, a contradiction. Hence $|Z(Q)|=3$ and $Q'=Z(Q)$ again.
\end{proof}

\begin{proposition}\label{Pr:LargeDerived}
Let $Q$ be a Bol loop of order $27$. If $Z(Q)=1$ then $|Q'|=9$.
\end{proposition}
\begin{proof}
We have $1<Q'$, else $Q$ is an abelian group. We have $Q'<Q$ by Theorem \ref{Nagy}. If $|Q'|=3$ then $|Z(Q)|=3$ by Corollary \ref{Cr:Likewise}, a contradiction. Hence $|Q'|=9$.
\end{proof}

\section{Bol loops: Computational results}

\subsection{Nontrivial center}

Centrally nilpotent Bol loops of order $27$ are easy to handle computationally. In the \texttt{GAP} package \texttt{RightQuasigroups}, one can compute all (small) central extensions of a cyclic group $Z=C_p$ of prime order by a given loop $F$ in a given variety of loops containing abelian groups. The algorithm is based on cocycles, coboundaries and the action of the group $\aut{Z}\times\aut{F}$ on the space of cocycles. The method is described in detail (for the case of Moufang loops) in \cite{NV64}.

Given a prime $p$, the code
\begin{verbatim}
        LoadPackage( "RightQuasigroups" );
        C := CyclicGroup( p );
        F := AsLoop( DirectProduct( C, C ) );
        basis := [ "((x*y)*z)*y = x*((y*z)*y)" ];
        extensions := AllLoopCentralExtensionsInVariety( F, p, basis );
        loops := LoopsUpToIsomorphism( extensions );
\end{verbatim}
constructs all right Bol loops that are central extensions of $C_p$ by $C_p\times C_p$ up to isomorphism.

For $p=3$, it returns $12$ right Bol loops of order $27$, namely $4$ groups and the $8$ nonassociative right Bol loops $B_1,\dots,B_8$ (previously obtained by Moorhouse \cite{Moorhouse} using a different algorithm). The calculation takes a fraction of a second.

For $p=5$, it returns $14$ right Bol loops of order $125$, namely $4$ groups and $10$ nonassociative right Bol loops. The calculation takes less than $5$ seconds.

For $p=7$, it returns $16$ right Bol loops of order $343$, namely $4$ groups and $12$ nonassociative right Bol loops. The calculation takes about $2$ minutes.

Note that in all cases the algorithm (correctly) does not find the cyclic group of order $p^3$. Combining these results with Theorem \ref{Th:Structure}, we have:

\begin{theorem}\label{Th:Nilp}
There are $13$ (resp. $15$, $17$) centrally nilpotent right Bol loops of order $3^3$ (resp. $5^3$, $7^3$).
\end{theorem}

The multiplication tables of the right Bol loops of Theorem \ref{Th:Nilp} can be downloaded from the website of the third author.

\begin{problem}
Determine the number of centrally nilpotent right Bol loops of order $p^3$ for every prime $p$.
\end{problem}

\subsection{Trivial center}\label{Ss:TrivialCenter}

In view of Theorem \ref{Th:Structure}, it remains to classify the Bol loops of order $27$ with trivial center. We do this in a roundabout way, taking advantage of the results from \cref{Sc:Theory} on $Q'$, $\lnuc{Q}$ and normal subloops of order $3$.

Let $Q$ be a right Bol loop of order $27$ with $Z(Q)=1$. By Proposition \ref{Pr:LargeDerived}, $|Q'|=9$. By Proposition \ref{Pr:LeftNuc}, $Q'\cap\lnuc{Q} > 1$. Thus there exists a subloop $N\leq\lnuc{Q}\cap Q'$ such that $|N|=3$.

We set up a \texttt{mace4} search for all right Bol loops of order $27$ containing a subloop $N\le\lnuc{Q}$ of order $3$ such that $N$ is contained in a normal subloop $M$ of order $9$. The search was split into two cases: $M\cong\Z_3\times\Z_3$ and $M\cong\Z_9$. Both searches ran to completion in approximately 20 minutes. The elementary abelian case generated $177$ models and the cyclic case generated $87$ models. The output of each search was imported into \texttt{GAP}. Using the \texttt{RightQuasigroups} package, each case's loops were sorted up to isomorphism with a representative loop extracted from each isomorphism class. The elementary abelian case yielded $11$ loops, $8$ of which were nonassociative. The cyclic case yielded $12$ loops, $8$ of which were nonassociative. Merging the results yielded $10$ nonassociative loops, namely the loops $B_1,\dots,B_{10}$.

\medskip

This completes the classification of Bol loops of order $27$.

\begin{theorem}\label{Th:Main}
There are $15$ right Bol loops of order $27$ up to isomorphism, including five groups.
\end{theorem}

The ten nonassociative right Bol loops from \cref{Th:Main} are constructed in Section \ref{Sc:Constructions}.

\section*{Acknowledgment}

We thank the referee for several useful suggestions that streamlined the paper, particularly in Subsection \ref{Ss:TrivialCenter}.

\end{document}